\documentclass[12pt,fleqn]{article}

\usepackage{amscd,amsmath,amssymb,amsthm}
\usepackage{geometry} 
\usepackage[pdfpagemode=UseNone,pdfstartview=FitH,hypertex]{hyperref}

\newcommand{\bdry}[1]{\partial #1}
\newcommand{\D}{\mathcal D}
\newcommand{\eps}{\varepsilon}
\newcommand{\F}{\mathcal F}
\newcommand{\N}{\mathbb N}
\newcommand{\norm}[2][]{\left\|#2\right\|_{#1}}
\renewcommand{\o}{\text{o}}
\newcommand{\PS}[1]{$(\text{PS})_{#1}$}
\newcommand{\pnorm}[2][]{\if #1'' \left|#2\right|_p \else \left|#2\right|_{#1} \fi}
\newcommand{\R}{\mathbb R}
\newcommand{\seq}[1]{\left(#1\right)}
\newcommand{\set}[1]{\left\{#1\right\}}
\newcommand{\vol}[1]{\left|#1\right|}
\newcommand{\Z}{\mathbb Z}

\newtheorem{corollary}{Corollary}[section]
\newtheorem{lemma}[corollary]{Lemma}

\newtheorem{theorem}[corollary]{Theorem}

\theoremstyle{definition}

\theoremstyle{remark}

\numberwithin{equation}{section}

\title{\bf A bifurcation and multiplicity result for a critical growth elliptic problem\thanks{{\em MSC2010:} Primary 35J92, Secondary 35B33, 35B32, 35J20
\newline \indent\; {\em Key Words and Phrases:} critical growth $p$-Laplacian problems, multiple nontrivial solutions, variational methods}}
\author{\bf Said El Manouni\\
Department of Mathematics\\
Imam Mohammad Ibn Saud Islamic University (IMSIU)\\
P. O. Box 90950, Riyadh 11623, Saudi Arabia\\
\em samanouni@imamu.edu.sa\\
[\bigskipamount]
\bf Kanishka Perera\\
Department of Mathematics\\
Florida Institute of Technology\\
150 W University Blvd, Melbourne, FL 32901-6975, USA\\
\em kperera@fit.edu}
\date{}

\begin{document}

\maketitle

\begin{abstract}
We consider a Br{\'e}zis-Nirenberg type critical growth $p$-Laplacian problem involving a parameter $\mu > 0$ in a smooth bounded domain $\Omega$. We prove the existence of multiple nontrivial solutions if either $\mu$ or the volume of $\Omega$ is sufficiently small. The proof is based on an abstract critical point theorem that only assumes a local \PS{} condition. Our results are new even in the semilinear case $p = 2$.
\end{abstract}

\section{Introduction}

Consider the critical growth $p$-Laplacian problem
\begin{equation} \label{1}
\left\{\begin{aligned}
- \Delta_p\, u + \mu\, |u|^{q-2}\, u & = \lambda\, |u|^{p-2}\, u + |u|^{p^\ast - 2}\, u && \text{in } \Omega\\[10pt]
u & = 0 && \text{on } \bdry{\Omega},
\end{aligned}\right.
\end{equation}
where $\Omega$ is a smooth bounded domain in $\R^N,\, N \ge 2$, $1 < q < p < N$, $p^\ast = Np/(N - p)$ is the critical Sobolev exponent, and $\lambda, \mu > 0$ are parameters. When $\mu = 0$ this problem reduces to the well-known Br{\'e}zis-Nirenberg problem
\[
\left\{\begin{aligned}
- \Delta_p\, u & = \lambda\, |u|^{p-2}\, u + |u|^{p^\ast - 2}\, u && \text{in } \Omega\\[10pt]
u & = 0 && \text{on } \bdry{\Omega},
\end{aligned}\right.
\]
which has been extensively studied in the literature beginning with the celebrated paper \cite{MR709644} (see, e.g., [1, 4--8, 12--16, 19, 20, 22, 23, 26, 29--33, 37--40]
for the semilinear case $p = 2$ and \cite{MR1974041,MR1741848,MR2885967,MR2504036,MR3912725,MR2514055,MR956567,MR912211,MR1009077,MR3469053} for the general case). However, as we will see in the present paper, a new phenomenon appears when $\mu > 0$, namely, if either $\mu$ or the volume of $\Omega$ is sufficiently small, then the number of solutions goes to infinity as $\lambda \to \infty$. A corresponding result does not seem to be known when $\mu = 0$ even in the semilinear case $p = 2$.

To state our result, let $\seq{\lambda_k}$ be the sequence of Dirichlet eigenvalues of the $p$-Laplacian in $\Omega$ defined using the $\Z_2$-cohomological index (see Perera \cite{MR1998432}). We will show that given any $\lambda \ge \lambda_k$, problem \eqref{1} has $k$ distinct pairs of nontrivial solutions if $\mu^{p^\ast/(p^\ast - q)} \vol{\Omega}$ is sufficiently small, where $\vol{\Omega}$ denotes the volume of $\Omega$. Let
\begin{equation} \label{6}
S = \inf_{u \in \D^{1,\,p}(\R^N) \setminus \set{0}}\, \frac{\int_{\R^N} |\nabla u|^p\, dx}{\left(\int_{\R^N} |u|^{p^\ast} dx\right)^{p/p^\ast}}
\end{equation}
be the best Sobolev constant. We have the following theorem.

\begin{theorem} \label{Theorem 1}
Assume that
\begin{equation} \label{2}
\mu^{p^\ast/(p^\ast - q)} \vol{\Omega} < \frac{pq\, S^{N/p}}{Np - (N - p)\, q}.
\end{equation}
Then for all $\lambda \ge \lambda_k$, problem \eqref{1} has $k$ distinct pairs of nontrivial solutions $\pm u^{\lambda,\,\mu}_j,\, j = 1,\dots,k$ such that $u^{\lambda,\,\mu}_j \to 0$ as $\mu \searrow 0$.
\end{theorem}

In particular, we have the following corollary for the case $\mu = 1$.

\begin{corollary}
If
\[
\vol{\Omega} < \frac{pq\, S^{N/p}}{Np - (N - p)\, q},
\]
then the problem
\[
\left\{\begin{aligned}
- \Delta_p\, u + |u|^{q-2}\, u & = \lambda\, |u|^{p-2}\, u + |u|^{p^\ast - 2}\, u && \text{in } \Omega\\[10pt]
u & = 0 && \text{on } \bdry{\Omega}
\end{aligned}\right.
\]
has $k$ distinct pairs of nontrivial solutions for all $\lambda \ge \lambda_k$. In particular, the number of solutions goes to infinity as $\lambda \to \infty$.
\end{corollary}

Theorem \ref{Theorem 1} and its corollary are new even in the semilinear case $p = 2$, which we restate as the following corollaries where $N \ge 3$, $1 < q < 2$, and $\seq{\lambda_k}$ is the sequence of Dirichlet eigenvalues of the Laplacian in $\Omega$.

\begin{corollary}
Assume that
\[
\mu^{2^\ast/(2^\ast - q)} \vol{\Omega} < \frac{2q\, S^{N/2}}{2N - (N - 2)\, q}.
\]
Then for all $\lambda \ge \lambda_k$, the problem
\[
\left\{\begin{aligned}
- \Delta u + \mu\, |u|^{q-2}\, u & = \lambda u + |u|^{2^\ast - 2}\, u && \text{in } \Omega\\[10pt]
u & = 0 && \text{on } \bdry{\Omega}
\end{aligned}\right.
\]
has $k$ distinct pairs of nontrivial solutions $\pm u^{\lambda,\,\mu}_j,\, j = 1,\dots,k$ such that $u^{\lambda,\,\mu}_j \to 0$ as $\mu \searrow 0$.
\end{corollary}

\begin{corollary}
If
\[
\vol{\Omega} < \frac{2q\, S^{N/2}}{2N - (N - 2)\, q},
\]
then the problem
\[
\left\{\begin{aligned}
- \Delta u + |u|^{q-2}\, u & = \lambda u + |u|^{2^\ast - 2}\, u && \text{in } \Omega\\[10pt]
u & = 0 && \text{on } \bdry{\Omega}
\end{aligned}\right.
\]
has $k$ distinct pairs of nontrivial solutions for all $\lambda \ge \lambda_k$. In particular, the number of solutions goes to infinity as $\lambda \to \infty$.
\end{corollary}

The variational functional associated with problem \eqref{1} is given by
\[
E(u) = \frac{1}{p} \int_\Omega |\nabla u|^p\, dx + \frac{\mu}{q} \int_\Omega |u|^q\, dx - \frac{\lambda}{p} \int_\Omega |u|^p\, dx - \frac{1}{p^\ast} \int_\Omega |u|^{p^\ast} dx, \quad u \in W^{1,\,p}_0(\Omega).
\]
Although this functional has $\Z_2$ symmetry, the existence of multiple critical points does not follow from standard arguments involving various index theories due to the lack of compactness. The functional $E$ satisfies the \PS{c} condition only for $c < \frac{1}{N}\, S^{N/p}$ (see Lemma \ref{Lemma 1}), and one needs to show that any minimax levels defined via an index theory are below this threshold for compactness in order to conclude that they are critical levels. We will prove Theorem \ref{Theorem 1} using an abstract critical point theorem recently obtained in Perera \cite{Pe23} that only assumes a local \PS{} condition and therefore can produce such critical levels (see Theorem \ref{Theorem 3}).

We would like to point out that it is somewhat surprising that the term $\mu\, |u|^{q-2}\, u$ results in more solutions since the term $\frac{\mu}{q} \int_\Omega |u|^q\, dx$ increases the energy without lowering the threshold for compactness.

\section{Proof of Theorem \ref{Theorem 1}}

In this section we prove Theorem \ref{Theorem 1}. The proof will be based on an abstract result from Perera \cite{Pe23}. To state this result, let $W$ be a Banach space and let $E \in C^1(W,\R)$ be an even functional, i.e., $E(-u) = E(u)$ for all $u \in E$. Assume that there exists $c^\ast > 0$ such that for all $c \in (0,c^\ast)$, $E$ satisfies the \PS{c} condition, i.e., every sequence $\seq{u_j}$ in $W$ such that $E(u_j) \to c$ and $E'(u_j) \to 0$ has a strongly convergent subsequence. Let $S = \set{u \in W : \norm{u} = 1}$ be the unit sphere in $W$ and let $i(A)$ denote the $\Z_2$-cohomological index of the symmetric set $A \subset W \setminus \set{0}$ (see Fadell and Rabinowitz \cite{MR0478189}).

\begin{theorem}[Perera {\cite[Theorem 2.1]{Pe23}}] \label{Theorem 3}
Let $C$ be a compact symmetric subset of $S$ with $i(C) = k \ge 1$. Assume that the origin is a strict local minimizer of $E$ with $E(0) = 0$ and there exists $R > 0$ such that
\begin{equation} \label{5}
\sup_{u \in A}\, E(u) \le 0, \qquad \sup_{u \in X}\, E(u) < c^\ast,
\end{equation}
where $A = \set{Ru : u \in C}$ and $X = \set{tu : u \in C,\, 0 \le t \le R}$. Then $E$ has $k$ distinct pairs of nontrivial critical points $\pm u_j,\, j = 1,\dots,k$ satisfying
\[
0 < E(u_j) \le \sup_{u \in X}\, E(u).
\]
\end{theorem}

Dirichlet eigenvalues of the $p$-Laplacian in $\Omega$ coincide with critical values of the functional
\[
\Psi(u) = \frac{1}{\int_\Omega |u|^p\, dx}, \quad u \in S = \set{u \in W^{1,\,p}_0(\Omega) : \norm{u} = 1}.
\]
Let $\F$ be the class of symmetric subsets of $S$, let $\F_k = \set{M \in \F : i(M) \ge k}$, and set
\[
\lambda_k := \inf_{M \in \F_k}\, \sup_{u \in M}\, \Psi(u), \quad k \in \N.
\]
Then $\seq{\lambda_k}$ is an unbounded sequence of eigenvalues (see Perera \cite{MR1998432}). Moreover, if $\lambda_k < \lambda_{k+1}$, then the sublevel set
\[
\Psi^{\lambda_k} = \set{u \in S : \Psi(u) \le \lambda_k}
\]
has a compact symmetric subset of index $k$ (see Degiovanni and Lancelotti \cite{MR2514055}). We will prove Theorem \ref{Theorem 1} by applying Theorem \ref{Theorem 3} with this set as $C$.

Solutions of problem \eqref{1} coincide with critical points of the $C^1$-functional
\[
E(u) = \frac{1}{p} \int_\Omega |\nabla u|^p\, dx + \frac{\mu}{q} \int_\Omega |u|^q\, dx - \frac{\lambda}{p} \int_\Omega |u|^p\, dx - \frac{1}{p^\ast} \int_\Omega |u|^{p^\ast} dx, \quad u \in W^{1,\,p}_0(\Omega).
\]
First we prove the following local \PS{} condition.

\begin{lemma} \label{Lemma 1}
The functional $E$ satisfies the {\em \PS{c}} condition for $c < \frac{1}{N}\, S^{N/p}$.
\end{lemma}

\begin{proof}
Let $\seq{u_j}$ be a sequence in $W^{1,\,p}_0(\Omega)$ such that
\begin{equation} \label{12}
E(u_j) = \frac{1}{p} \int_\Omega |\nabla u_j|^p\, dx + \frac{\mu}{q} \int_\Omega |u_j|^q\, dx - \frac{\lambda}{p} \int_\Omega |u_j|^p\, dx - \frac{1}{p^\ast} \int_\Omega |u_j|^{p^\ast} dx = c + \o(1)
\end{equation}
and
\begin{multline} \label{13}
E'(u_j)\, v = \int_\Omega |\nabla u_j|^{p - 2}\, \nabla u_j \cdot \nabla v\, dx + \mu \int_\Omega |u_j|^{q - 2}\, u_j\, v\, dx - \lambda \int_\Omega |u_j|^{p - 2}\, u_j\, v\, dx\\[7.5pt]
- \int_\Omega |u_j|^{p^\ast - 2}\, u_j\, v\, dx = \o(\norm{v}) \quad \forall v \in W^{1,\,p}_0(\Omega).
\end{multline}

First we show that $\seq{u_j}$ is bounded. Taking $v = u_j$ in \eqref{13} gives
\begin{equation} \label{14}
\int_\Omega |\nabla u_j|^p\, dx + \mu \int_\Omega |u_j|^q\, dx - \lambda \int_\Omega |u_j|^p\, dx - \int_\Omega |u_j|^{p^\ast} dx = \o(\norm{u_j}).
\end{equation}
Let $r \in (p,p^\ast)$. Subtracting \eqref{14} divided by $r$ from \eqref{12} gives
\begin{multline*}
\left(\frac{1}{p} - \frac{1}{r}\right) \int_\Omega |\nabla u_j|^p\, dx + \mu \left(\frac{1}{q} - \frac{1}{r}\right) \int_\Omega |u_j|^q\, dx - \lambda \left(\frac{1}{p} - \frac{1}{r}\right) \int_\Omega |u_j|^p\, dx\\[7.5pt]
+ \left(\frac{1}{r} - \frac{1}{p^\ast}\right) \int_\Omega |u_j|^{p^\ast} dx = c + \o(1) + \o(\norm{u_j}).
\end{multline*}
Since $1 < q < p < r < p^\ast$, this implies that $\norm{u_j}$ is bounded.

Passing to a subsequence, we may now assume that $\seq{u_j}$ converges to some $u$ weakly in $W^{1,\,p}_0(\Omega)$, strongly in $L^s(\Omega)$ for all $s \in [1,p^\ast)$, and a.e.\! in $\Omega$. Set $\widetilde{u}_j = u_j - u$. We will show that if, for some $\eps_0 > 0$, $\norm{\widetilde{u}_j} \ge \eps_0$, then $c \ge \frac{1}{N}\, S^{N/p}$, contrary to assumption, and conclude that $u_j \to u$ in $W^{1,\,p}_0(\Omega)$ for a further subsequence.

Since $u_j \to u$ in $L^p(\Omega)$ and in $L^q(\Omega)$, \eqref{12} and \eqref{14} reduce to
\begin{equation} \label{15}
\frac{1}{p} \int_\Omega |\nabla u_j|^p\, dx + \frac{\mu}{q} \int_\Omega |u|^q\, dx - \frac{\lambda}{p} \int_\Omega |u|^p\, dx - \frac{1}{p^\ast} \int_\Omega |u_j|^{p^\ast} dx = c + \o(1)
\end{equation}
and
\begin{equation} \label{16}
\int_\Omega |\nabla u_j|^p\, dx + \mu \int_\Omega |u|^q\, dx - \lambda \int_\Omega |u|^p\, dx - \int_\Omega |u_j|^{p^\ast} dx = \o(1),
\end{equation}
respectively. Moreover, taking $v = u$ in \eqref{13} and passing to the limit gives
\begin{equation} \label{17}
\int_\Omega |\nabla u|^p\, dx + \mu \int_\Omega |u|^q\, dx - \lambda \int_\Omega |u|^p\, dx - \int_\Omega |u|^{p^\ast} dx = 0.
\end{equation}
We have
\begin{gather}
\notag \int_\Omega |\nabla u_j|^p\, dx - \int_\Omega |\nabla u|^p\, dx = \int_\Omega |\nabla \widetilde{u}_j|^p\, dx + \o(1),\\[7.5pt]
\notag \int_\Omega |u_j|^{p^\ast} dx - \int_\Omega |u|^{p^\ast} dx = \int_\Omega |\widetilde{u}_j|^{p^\ast} dx + \o(1)
\end{gather}
by the Br{\'e}zis-Lieb lemma \cite[Theorem 1]{MR699419}, so subtracting \eqref{17} from \eqref{16} gives
\[
\int_\Omega |\nabla \widetilde{u}_j|^p\, dx = \int_\Omega |\widetilde{u}_j|^{p^\ast} dx + \o(1).
\]
Combining this with $\eqref{6}$ and using $\norm{\widetilde{u}_j} \ge \eps_0$, we get
\begin{equation} \label{18}
\int_\Omega |\nabla \widetilde{u}_j|^p\, dx \ge S^{N/p} + \o(1).
\end{equation}
On the other hand, subtracting \eqref{16} divided by $p^\ast$ from \eqref{15} gives
\[
\frac{1}{N} \int_\Omega |\nabla u_j|^p\, dx + \mu \left(\frac{1}{q} - \frac{1}{p^\ast}\right) \int_\Omega |u|^q\, dx - \frac{\lambda}{N} \int_\Omega |u|^p\, dx = c + \o(1),
\]
and subtracting \eqref{17} divided by $N$ from this gives
\[
\frac{1}{N} \int_\Omega |\nabla \widetilde{u}_j|^p\, dx + \mu \left(\frac{1}{q} - \frac{1}{p}\right) \int_\Omega |u|^q\, dx + \frac{1}{N} \int_\Omega |u|^{p^\ast} dx = c + \o(1).
\]
Since $p > q$, this together with \eqref{18} implies that $c \ge \frac{1}{N}\, S^{N/p}$.
\end{proof}

We are now ready to prove Theorem \ref{Theorem 1}.

\begin{proof}[Proof of Theorem \ref{Theorem 1}]
In view of Lemma \ref{Lemma 1}, we apply Theorem \ref{Theorem 3} with $c^\ast = \frac{1}{N}\, S^{N/p}$. First we show that the origin is a strict local minimizer of $E$. We have the Gagliardo-Nirenberg inequality
\[
\int_\Omega |u|^p\, dx \le C \left(\int_\Omega |\nabla u|^p\, dx\right)^\theta \left(\int_\Omega |u|^q\, dx\right)^{(1 - \theta)\, p/q},
\]
where $C > 0$ is a constant and $\theta = N(p - q)/[pq + N(p - q)] \in (0,1)$. Combining this with the Young's inequality shows that given any $\eps > 0$, there exists a constant $C_\eps > 0$ such that
\[
\int_\Omega |u|^p\, dx \le \eps \int_\Omega |\nabla u|^p\, dx + C_\eps \left(\int_\Omega |u|^q\, dx\right)^{p/q}.
\]
This together with \eqref{6} gives
\begin{multline*}
E(u) \ge \frac{1}{p}\, (1 - \lambda \eps) \int_\Omega |\nabla u|^p\, dx - \frac{1}{p^\ast\, S^{p^\ast/p}} \left(\int_\Omega |\nabla u|^p\, dx\right)^{p^\ast/p}\\[7.5pt]
+ \frac{\mu}{q} \int_\Omega |u|^q\, dx - \frac{\lambda C_\eps}{p} \left(\int_\Omega |u|^q\, dx\right)^{p/q}.
\end{multline*}
Since $\eps > 0$ is arbitrary and $p^\ast > p > q$, the desired conclusion follows from this.

By increasing $k$ if necessary, we may assume that $\lambda_k < \lambda_{k+1}$. Then the sublevel set $\Psi^{\lambda_k}$ has a compact symmetric subset $C$ of index $k$ (see Degiovanni and Lancelotti \cite{MR2514055}). For $u \in C$,
\begin{equation} \label{100}
\int_\Omega |\nabla u|^p\, dx = 1, \qquad \int_\Omega |u|^p\, dx \ge \frac{1}{\lambda_k} \ge \frac{1}{\lambda}.
\end{equation}
By the H\"{o}lder inequality,
\[
\int_\Omega |u|^q\, dx \le \vol{\Omega}^{1 - q/p^\ast}\! \left(\int_\Omega |u|^{p^\ast} dx\right)^{q/p^\ast}.
\]
So for $u \in C$ and $t \ge 0$,
\begin{multline} \label{7}
E(tu) = \frac{t^p}{p} \left(\int_\Omega |\nabla u|^p\, dx - \lambda \int_\Omega |u|^p\, dx\right) + \frac{\mu t^q}{q} \int_\Omega |u|^q\, dx - \frac{t^{p^\ast}}{p^\ast} \int_\Omega |u|^{p^\ast} dx\\[7.5pt]
\le \frac{\mu \vol{\Omega}^{1 - q/p^\ast}}{q} \left(t^{p^\ast} \int_\Omega |u|^{p^\ast} dx\right)^{q/p^\ast} - \frac{t^{p^\ast}}{p^\ast} \int_\Omega |u|^{p^\ast} dx,
\end{multline}
and maximizing the last expression over all $t \ge 0$ gives
\begin{equation} \label{8}
\sup_{u \in C,\, t \ge 0}\, E(tu) \le \left(\frac{1}{q} - \frac{1}{p^\ast}\right)\! \vol{\Omega} \mu^{p^\ast/(p^\ast - q)}.
\end{equation}
Let $R > 0$ and let $A$ and $X$ be as in Theorem \ref{Theorem 3}. Then \eqref{7} gives
\[
E(Ru) \le \frac{\mu \vol{\Omega}^{1 - q/p^\ast}\! R^q}{q} \left(\int_\Omega |u|^{p^\ast} dx\right)^{q/p^\ast} - \frac{R^{p^\ast}}{p^\ast} \int_\Omega |u|^{p^\ast} dx \quad \forall u \in C.
\]
Since
\[
\int_\Omega |u|^{p^\ast} dx \ge \frac{1}{\vol{\Omega}^{p^\ast/p - 1}} \left(\int_\Omega |u|^p\, dx\right)^{p^\ast/p} \ge \frac{1}{\vol{\Omega}^{p^\ast/p - 1} \lambda_k^{p^\ast/p}}
\]
for $u \in C$ by the H\"{o}lder inequality and \eqref{100}, this implies that
\[
\sup_{u \in A}\, E(u) \le 0
\]
if $R$ is sufficiently large. On the other hand, \eqref{8} gives
\[
\sup_{u \in X}\, E(u) \le \left(\frac{1}{q} - \frac{1}{p^\ast}\right)\! \vol{\Omega} \mu^{p^\ast/(p^\ast - q)} < \frac{1}{N}\, S^{N/p}
\]
by \eqref{2}. So $E$ has $k$ distinct pairs of nontrivial critical points $\pm u^{\lambda,\,\mu}_j,\, j = 1,\dots,k$ satisfying
\begin{equation} \label{9}
0 < E(u^{\lambda,\,\mu}_j) \le \left(\frac{1}{q} - \frac{1}{p^\ast}\right)\! \vol{\Omega} \mu^{p^\ast/(p^\ast - q)}.
\end{equation}

It only remains to show that each $u^{\lambda,\,\mu}_j \to 0$ as $\mu \searrow 0$. We have
\begin{equation} \label{10}
E(u^{\lambda,\,\mu}_j) = \frac{1}{p} \int_\Omega |\nabla u^{\lambda,\,\mu}_j|^p\, dx + \frac{\mu}{q} \int_\Omega |u^{\lambda,\,\mu}_j|^q\, dx - \frac{\lambda}{p} \int_\Omega |u^{\lambda,\,\mu}_j|^p\, dx - \frac{1}{p^\ast} \int_\Omega |u^{\lambda,\,\mu}_j|^{p^\ast} dx \to 0
\end{equation}
as $\mu \searrow 0$ by \eqref{9} and
\begin{equation} \label{11}
E'(u^{\lambda,\,\mu}_j)\, u^{\lambda,\,\mu}_j = \int_\Omega |\nabla u^{\lambda,\,\mu}_j|^p\, dx + \mu \int_\Omega |u^{\lambda,\,\mu}_j|^q\, dx - \lambda \int_\Omega |u^{\lambda,\,\mu}_j|^p\, dx - \int_\Omega |u^{\lambda,\,\mu}_j|^{p^\ast} dx = 0.
\end{equation}
Subtracting \eqref{11} divided by $p$ from \eqref{10} gives
\[
\mu \left(\frac{1}{q} - \frac{1}{p}\right) \int_\Omega |u^{\lambda,\,\mu}_j|^q\, dx + \frac{1}{N} \int_\Omega |u^{\lambda,\,\mu}_j|^{p^\ast} dx \to 0,
\]
so $\int_\Omega |u^{\lambda,\,\mu}_j|^{p^\ast} dx \to 0$. Then $\int_\Omega |u^{\lambda,\,\mu}_j|^p\, dx \to 0$ by the H\"{o}lder inequality, so $\int_\Omega |\nabla u^{\lambda,\,\mu}_j|^p\, dx \to 0$ by \eqref{11}.
\end{proof}

\section*{Acknowledgements}

This work was supported and funded by the Deanship of Scientific Research at Imam Mohammad Ibn Saud Islamic University (IMSIU) (grant number IMSIU-RP23037).

\def\cprime{$''$}

\end{document}